\documentclass{amsart}
\usepackage{graphicx}
\usepackage[all]{xy}
\usepackage[ansinew]{inputenc}
\usepackage{amsmath}
\usepackage{amscd}
\usepackage{amssymb}
\usepackage{latexsym}
\usepackage[active]{srcltx}
\usepackage{mathrsfs}
\usepackage{color}

\newtheorem{thm}{Theorem}

\theoremstyle{definition}
\newtheorem{cor}[thm]{Corollary}
\newtheorem{lem}[thm]{Lemma}

\newtheorem{fact}[thm]{Fact}

\newtheorem*{thmA}{Theorem A}
\newtheorem*{thmB}{Theorem B}
\newtheorem*{thmC}{Theorem C}
\newtheorem*{thmD}{Theorem D}
\newtheorem*{thmE}{Theorem E}
\newtheorem*{thmF}{Theorem F}

\numberwithin{equation}{section}

\newcommand{\N}{\mathbb{N}}
\newcommand{\Z}{\mathbb{Z}}
\newcommand{\Q}{\mathbb{Q}}
\newcommand{\R}{\mathbb{R}}

\newcommand{\Cal}{\mathcal}

\def \<{\langle}
\def \>{\rangle}

\def \((  {(\!(}
\def \)) {)\!)}

\begin{document}

\title[]{Expansions of subfields of the real field by a discrete set}%
\author{Philipp Hieronymi}
\address{University of Illinois at Urbana-Champaign\\
Department of Mathematics\\
1409 W. Green Street\\
Urbana, IL 61801\\
USA}
\email{P@hieronymi.de}


\maketitle

\begin{abstract} Let $K$ be a subfield of the real field, $D\subseteq K$ be a discrete set and $f:D^n \to K$ be such that $f(D^n)$ is somewhere dense. Then $(K,f)$ defines $\mathbb{Z}$. We present several applications of this result. We show that $K$ expanded by predicates for different cyclic multiplicative subgroups defines $\Z$. Moreover, we prove that every definably complete expansion of a subfield of the real field satisfies an analogue of the Baire Category Theorem.
\end{abstract}

\section{Introduction}
Let $K$ be a subfield of the field of real numbers.

\begin{thmA} Let $D\subseteq K$ be discrete, $n \in \N$ and let $f:D^n\to K$ be such that $f(D)$ is somewhere dense. Then $(K,f)$ defines $\Z$.
\end{thmA}
A set is somewhere dense if its topological closure has interior. This result generalizes earlier work of the author in \cite{discrete} where Theorem A is shown in the case that $K=\R$ and $D$ is closed and discrete. The proof in \cite{discrete} relies crucially on the topological completeness of $\R$ and hence does not work for subfields of the real field. One can even construct a subfield $K$ and a function $f:D\to K$ that satisfy the assumptions of Theorem A, but the parameter-free formula that defines $\Z$ in $(\R,f)$ does not define $\Z$ in $(K,f)$. The work in the current paper shows how results from \cite{discrete} can still be used to establish Theorem A.\\

A subset of a subfield $K$ of $\R$ is discrete in the induced topology on $K$ if and only if it is discrete in the order topology on $\R$. However, there are discrete subsets $D$ of $K$ that are closed in the induced topology on $K$, but that are not closed in $\R$. Such discrete sets may even fail to be well ordered by the ordering on $\R$. To make use of the results of \cite{discrete} we establish the following theorem.

\begin{thmB}\label{thmdisc} Let $D\subseteq K$ be a discrete set. Then there is a discrete set $E\subseteq K$ such that $E$ is closed in $\R$, $(K,D)$ and $(K,E)$ are interdefinable and there is a surjection $g: E\to D$ definable in $(K,D)$.
\end{thmB}

The proof of Theorem A will be presented in the section 4. In section 2 we prove a generalization of Miller's Lemma on Asymptotic Extraction of Groups from \cite{proj} that plays a key role in the proof of Theorem A. Section 3 gives a proof of Theorem B. In the rest of this section, several applications of Theorem A and B will be discussed.

\subsection*{Two discrete subgroups} For any $\alpha \in K^{\times}$, let
\begin{equation*}
\alpha^{\mathbb{Z}}:= \{\alpha^k : k \in \mathbb{Z}\}.
\end{equation*}
 In \cite{lou} van den Dries established that the
structure $(K,\alpha^{\mathbb{Z}})$ is model theoretically tame, when $K$ is a real closed field subfield of the real field. In particular, he showed that $\Z$ is not definable in that structure. Theorem A allows us to show that this is not the case in the structure $(K,\alpha^{\Z},\beta^{\Z})$.

\begin{thmC} Let $\alpha,\beta \in K_{>0}$ with $\log_{\alpha}(\beta)\notin \Q$. Then $(K,\alpha^{\Z},\beta^{\Z})$ defines $\Z$.
\end{thmC}
\begin{proof} The set $\alpha^{\Z}\cup \beta^{\Z}$ is discrete and definable in $(K,\alpha^{\Z},\beta^{\Z})$. Let $g: K_{>0}\times K_{>0} \to K$ be the function mapping $(a,b)$ to $ab$. The image of $(\alpha^{\Z}\cup \beta^{\Z})\times (\alpha^{\Z}\cup \beta^{\Z})$ under $g$ is $\alpha^{\Z}\beta^{\Z}$ and hence dense in $K_{>0}$. Hence $(K,\alpha^{\Z},\beta^{\Z})$ defines $\Z$ by Theorem A.
\end{proof}

%

\subsection*{An analogue of the Baire Category Theorem} An expansion $\mathcal{K}$ of $K$ is \linebreak \emph{definably complete} if every bounded subset of $K$ definable in $\mathcal{K}$ has a supremum in $K$. For details, see Miller \cite{ivp}. Given a subset $Y$ of $K^2$ and $a \in K$, we denote $\{b : (b,a) \in Y\}$ by $Y_a$.

\begin{thmD} Let $\mathcal{K}$ be a definably complete expansion of $K$. Then $\mathcal{K}$ is definably Baire; that is there exists \emph{no} set $Y \subseteq K_{>0}^2$ definable in $\mathcal{K}$ such that
\begin{itemize}
\item[(i)] $Y_t$ is nowhere dense for all $t\in K_{>0}$,
\item[(ii)]$Y_s \subseteq Y_t$ for all $s,t\in K_{>0}$ with $s<t$, and
\item[(iii)] $\bigcup_{t \in K_{>0}}Y_t = K$.
\end{itemize}
\end{thmD}
\begin{proof} Suppose $\mathcal{K}$ is not definably Baire. By \cite[Corollary 6.6]{fornasiero2}, there is a closed and discrete set $D\subseteq K$ definable in $\mathcal{K}$ and $f: D \to K$ definable in $\mathcal{K}$ such that the image of $f$ is dense in $K$. By Theorem A, $\Z$ is definable in $\mathcal{K}$. Thus $\mathcal{K}$ is definably Baire by \cite[Lemma 6.2]{fornasiero2}.
\end{proof}
Definable versions of standard facts from real analysis hold in definably complete expansions of ordered fields that satisfy the conclusion of Theorem D. For details, see the work of Fornasiero and Servi in \cite{fs}.

\subsection*{Optimality of dichotomies over $\R$} By Theorem B, the dichotomy in \cite[Theorem 1.2]{discrete} extends to discrete subsets of $\R$ as follows.

\begin{thmE} Let $\mathcal{R}$ be an o-minimal expansion of $\mathbb{R}$ and let $D \subseteq \mathbb{R}$ be discrete. Then either
\begin{itemize}
 \item  $(\mathcal{R},D)$ defines $\mathbb{Z}$ or
 \item  every  subset of $\mathbb{R}$ definable in $(\mathcal{R},D)$ has interior or is nowhere dense.
\end{itemize}
\end{thmE}

However, by the following result neither in Theorem E nor in \cite[Theorem 1.2]{discrete} can the statement \emph{`is nowhere dense'} be replaced by \emph{`is a finite union of discrete sets'}.

\begin{thmF} There is a closed and discrete set $D\subseteq \R$ such that $(\R,D)$ does not define $\Z$, but defines a set that is \emph{not} $F_{\sigma}$.
\end{thmF}
\begin{proof} By \cite[2.3]{interpret} there is a discrete set $D$ such that $(\R,D)$ does not define $\Z$, but sets that are not $F_{\sigma}$. By Theorem B, we can assume that $D$ is closed.
\end{proof}

\subsection*{Acknowledgements}
I would like to thank Harvey Friedman and Chris Miller for bringing these questions to my attention, and Lou van Dries and Antongiulio Fornasiero for helpful discussions on the topic of this paper.

\subsection*{Notation} In the rest of the paper $K$ will always be a fixed subfield of $\R$. As before, we do not distinguish between the field $K$ and its underlying set. We will use $a,b,c$ for elements of $K$. The letters $l,n,m,N$ will always denote natural numbers. When we say definable, we mean definable with parameters. Given a subset $A$ of $K^{n}\times K^m$ and $a \in K^m$, we denote the set $\{b : (b,a) \in A\}$ by $A_a$.

\section{Asymptotic extraction}

\begin{lem}\label{asym} Let $\mathcal{K}$ be an expansion of $K$ and let $S\subseteq K_{>0}\times K^l$ be definable in $\mathcal{K}$ such that for every $n\in \N$ and every $\varepsilon \in K\cap(0,1/2)$, there is $b \in K^l$ such that
\begin{itemize}
\item[(1)] $S_b \subseteq \bigcup_{m \in \N, m\leq n} (m-\varepsilon,m+\varepsilon)$, and
\item[(2)] $|S_b \cap (m-\varepsilon,m+\varepsilon)| = 1$ for $m \leq n$.
\end{itemize}
Then $\mathcal{K}$ defines $\Z$.
\end{lem}
\begin{proof} For $\varepsilon \in K_{>0}$ define $B_{\varepsilon}$ as the set of all $b \in K^l$ that satisfy the following two properties:
\begin{itemize}
\item[(i)] $|a_1-a_2| \geq 1 - \varepsilon$, for all $a_1,a_2 \in S_b$ with $a_1\neq a_2$ and
\item[(ii)] $|a_1-a_2| \leq 1 + \varepsilon$ for all $a_1,a_2 \in S_b$ with $S_b\cap (a_1,a_2)=\emptyset$.
\end{itemize}
For $b\in B_{\varepsilon}$, let $\lambda(b)$ be the smallest element of $S_b$. Such an element exists, since $S_b \subseteq K_{>0}$. Set
$$S'_b:= \{ a - \lambda(b) : a \in S_b\}.$$
Finally, define
$$
W := \{ c \in K : \forall \varepsilon \in K\cap(0,1/2) \exists b \in B_{\varepsilon} (c-\varepsilon,c+\varepsilon)\cap S_b'\neq \emptyset\}.
$$
We will finish the proof by showing that $W = \N$.\\
Let $n \in \N$ and $\varepsilon \in K\cap(0,1/2)$. By our assumption on $S$, there is $b\in K^l$ such that $$S_b \subseteq \bigcup_{m\leq n} (m-\frac{\varepsilon}{2},m+\frac{\varepsilon}{2}) \textrm{ and } |S_b \cap (m-\frac{\varepsilon}{2},m+\frac{\varepsilon}{2})| = 1$$ for $m \leq n$. Hence $|a_1 - a_2| \in (1- \varepsilon, 1+\varepsilon)$ for two adjacent elements $a_1,a_2 \in S_b$. Thus $b \in B_{\varepsilon}$. Since $\lambda(b) \in (0,\frac{\varepsilon}{2})$, we have that $|S_b' \cap (n-\varepsilon,n+\varepsilon)| = 1$. Hence $n\in W$.\\
Let $c \in K$ be such that $c\in(n,n+1)$ for some $n\in \N$. Let $\varepsilon \in K_{>0}$ be such that $2(n+1)\varepsilon \leq \min \{c-n,n+1-c\}$ and let $b \in B_{\varepsilon}$. Since $b \in B_{\varepsilon}$,
$$
S_b' \cap (n,n+1) \subseteq (n-n\varepsilon,n+n\varepsilon) \cap (n+1-(n+1)\varepsilon,n+1+(n+1)\varepsilon).
$$
Because of our choice of $\varepsilon$, we have $c-\varepsilon>n+n\varepsilon$ and $c+\varepsilon<n+1-(n+1)\varepsilon$. Hence $(c-\varepsilon,c+\varepsilon)\cap S_b'= \emptyset$ and $c \notin W$.
\end{proof}

Lemma \ref{asym} is a generalization of Miller's Lemma on Asymptotic Extraction of Groups from \cite[p.1484]{proj}.

\begin{cor} An expansion $\mathcal{K}$ of $K$ defines $\Z$ iff it defines the range of a sequence $(d_i)_{i\in \N}$ of elements in $K$ such that $\lim_{i \in \N} d_{i+1} - d_i \in \R\setminus \{0\}$.
\end{cor}
\begin{proof} Let $c\in \R\setminus \{0\}$ and $(d_i)_{i\in \N}$ be a sequence of elements in $K$ such that $\lim_{i\to \infty} d_{i+1} - d_{i} = c$. Let $D = \{d_i : i\in \N\}$ and $\Cal K$ be an expansion of $K$ that defines $D$. By exchanging the set $D$ by $\{-d : d \in D\}$, we can assume that $c>0$. Let $S \subseteq K_{>0} \times K^3$ be
$$
\{(a,b_1,b_2,b_3,b_4) \in K : b_1,b_2,b_3 \in D, b_1<b_2<b_3, a=\frac{b_2-b_1}{b_4}\}.
$$
Since $K$ is dense in $\R$ and for every $j \in \N$, $\lim_{i\to \infty} d_{i+j} - d_i = jc$, the set $S$ satisfies the assumptions (1) and (2) of Lemma \ref{asym}. Hence $\Cal R$ defines $\Z$.
\end{proof}

\section{Defining discrete sets that are closed in $\R$}

We say a set $X\subseteq K$ is \emph{closed in $\R$} if it is closed in the order topology on $\R$.

\begin{lem}\label{distance} Let $D\subseteq K_{>0}$ be discrete and closed in $\R$. There are $A\subseteq K_{>0}$ and a bijection $g:D\to A$ such that $g$ is definable in $(K,D)$ and $|a-b|\geq 1$ for all distinct $a,b\in A$.
\end{lem}
\begin{proof} Let $\sigma : D \to D$ be the successor function on the well-ordered set $(D,<)$. Define $g: D \to K$ by
$$
d \mapsto d \cdot \max \left( \{(\sigma(e)-e)^{-1} \ : \ e \in D, e < d \} \cup \{1\}\right).
$$
The maximum in the definition of $g$ always exists in $K$, because the set $\{e \in D : e < d\}$ is finite. The function $g$ is strictly increasing and definable in $(K,D)$. The image of $D$ under $g$ is discrete and closed as subset of $\R$. By construction, the distance between two elements of $g(D)$ is at least 1.
\end{proof}

\begin{lem}\label{closedef} Let $D\subseteq K_{>0}$ be an infinite discrete set. Then $(K,D)$ defines an infinite discrete set $A\subseteq K_{>0}$ that is closed in $\R$.
\end{lem}
\begin{proof} For every $\varepsilon \in K_{>0}$, we define\footnote{This definition was first used by Fornasiero in \cite[Remark 4.16]{fornasiero} for defining closed and discrete sets in definably complete expansions of fields.}
$$
B_{\varepsilon} := \{ d \in D : (d-\varepsilon,d+\varepsilon)\cap D = \{d\} \}.
$$
Note that $B_{\varepsilon}\supseteq B_{\delta}$, for $\varepsilon,\delta \in K_{>0}$ with $\varepsilon \leq \delta$.
If there is $\varepsilon \in K$ such that $B_{\varepsilon}$ is infinite, then this $B_{\varepsilon}$ is unbounded, discrete and closed in $\R$. So we can reduce to the case that $B_{\varepsilon}$ is finite for every $\varepsilon \in K$.\newline
Let $\varepsilon \in K_{>0}$ be such that $B_{\varepsilon}$ contains at least two elements. Let $g:(0,\varepsilon) \to D$ be the function that maps
$$
\delta \mapsto \max \left(\{ (d_1 - d_2)^{-1}, d_1-d_2 : d_1,d_2 \in B_{\delta}, d_1> d_2\} \cup \{1\}\right).
$$
Then $g(\left(0,\varepsilon\right))$ is infinite, since $D$ is. On the other hand, for every $\delta \in (0,\varepsilon)$, $g(\left(\delta,\varepsilon\right))$ is finite and $g\to \infty$ as $\delta \to 0^{+}$. Hence for every $N\in \N$, $\left(1,N\right)\cap g(\left(0,\varepsilon\right))$ is finite and thus $g(\left(0,\varepsilon\right))$ is closed in $\R$.
\end{proof}
\begin{proof}[Proof of Theorem B] Let $D$ be a discrete subset of $K$. By replacing $D$ by
$$
\{ -(d-1)^{-1} : d \in D_{\leq 0} \} \cup \{ 1 + d : d \in D_{>0}\},
$$
we can assume that $D \subseteq K_{>0}$. By Lemma \ref{distance} and \ref{closedef}, there is an infinite set $A\subseteq K_{>0}$ definable in $(K,D)$ such that $|a_1-a_2|\geq 1$ for all $a_1,a_2 \in A$ with $a_1\neq a_2$. Let $\sigma: A \to A$ be the successor function on the well ordered set $(A,<)$.
We now construct a discrete set $E$ that is closed in $\R$ and encodes all the information about $D$. For every $a \in A$, set
$$
B_a := \{ d \in D : d < a \textrm{ and } (d-a^{-1},d+a^{-1})\cap D = \{d\} \}.
$$
The set $B_a$ is finite and definable in $(K,D)$ for every $a\in A$. Moreover, $B_{a_1} \subseteq B_{a_2}$ for $a_1,a_2 \in A$ with $a_1\leq a_2$. Since $D$ is discrete, $D=\bigcup_{a \in A} B_a$.
Further for $a \in A$, define
$$
C_a := \{a + \frac{d}{a} : d \in B_a\}.
$$
Then $C_a$ is finite, definable in $(K,D)$ and $$C_a\subseteq (a,a+1) \subseteq (a,\sigma(a)).$$ Finally set $F := \bigcup_{a \in A} C_a$. Since $F \cap (a,\sigma(a))=C_a$ is finite for every $a\in A$, the set $F$ is discrete and closed in $\R$. Now define
$$
E := F \cup \{-a \ : \ a \in A\}.
$$
Then $E$ is discrete and closed in $\R$, since $A$ and $F$ are. Moreover, $A$ and $F$ are definable in $(K,E)$, because $A,F\subseteq K_{>0}$.\\
\noindent It is only left to show that there is a surjection $f: E \to D$ definable in $(K,E)$. Let $h : K \to A$ be a function mapping a real number $x$ to the largest $a \in A$ with $a<x$ if such an $a$ exists, and to 0 otherwise. Note that $h$ is definable in $(K,E)$, because $A$ is. Define a function $g: K \to K$ by
\begin{align*}
g(a):= h(a)(a-h(a)).
\end{align*}
The image of $C_a$ under $g$ is $B_a$ for each $a\in A$, because $C_a \subseteq (a,\sigma(a))$. Hence the image of $F$ under $g$ is $D$, since $F=\bigcup_{a \in A} C_a$. Let $d \in D$. Then let $f: E \to D$ be the function that maps $a\in E$ to $g(a)$ if $a \in F$, and to $d$ otherwise. The image of $E$ under $f$ is $D$ and $f$ is definable in $(K,E)$, since $g$ and $F$ are.
\end{proof}

\begin{lem}\label{newdisc} Let $D\subseteq K_{>0}$ be discrete and closed in $\R$. There are $E\subseteq K_{>0}$, $n\in \N$ and a bijection $g:D^n\to E$ such that $g$ is definable in $(K,D)$ and $E$ discrete  and closed in $\R$.
\end{lem}
\begin{proof} By Lemma \ref{distance}, we can assume that the distance between two elements of $D$ is at least 1.
Let $h: K_{>0} \times K^{n} \to K$ be defined by
$$
(x_0,x_1,...,x_n) \mapsto x_0 + \sum_{i=1}^{n} \frac{x_i}{(nx_0)^i}.
$$
Consider $g: D^n \to K$ defined by
$$
(d_1,...,d_n) \mapsto h(\max\{d_1,...,d_n\},d_1,...,d_n).
$$
It is easy to show that $g$ is injective and $g(D^n)$ is discrete and closed in $\R$.
\end{proof}

\section{Proof of Theorem A} Let $D$ be a discrete subset of $K$ and let $f:D^n \to K$ be a function such that $f(D^n)$ is somewhere dense. By Theorem B we can reduce to the case that $D$ is closed in $\R$. After a simple modification, we can assume that $D\subseteq K_{>0}$. By Lemma \ref{distance} and \ref{newdisc} we can assume that $n=1$ and that the distance between two distinct elements of $D$ is at least $1$. Composing $f$ with a semialgebraic function we can even assume that $f(D)\subseteq (1,2)$.\\

\noindent We recall several definitions from \cite{discrete}. Let $\varphi(x,y)$ be the formula
$$
\forall u \in D \left[f(u)<y<f(u)(1+u^{-2})\right]\rightarrow (u < x^{\frac{1}{7}} \vee u > x).
$$
Note that for all $a,b \in K$
$$
(\R,f)\models \varphi(a,b) \textrm{ iff } (K,f) \models \varphi(a,b).
$$
For $c \in \R$, define
$$
A_c :=\{ d \in D : f(d) < c < f(d)\cdot (1+d^{-2}) \wedge \varphi(d,c)\}.
$$
Further for $c \in \R$ let $v_c: D \setminus \{f^{-1}(c)\} \to \R$ be given by
$$
v_c(x) := \frac{x^{-2}f(x)}{c-f(x)}.
$$
The following Fact is the key step in the proof of \cite[Theorem 1.1]{discrete}.
\begin{fact}{(see \cite[p. 2166]{discrete})}\label{factdisc} There is an increasing sequence $(d_n)_{n=1}^{\infty}$ of elements in $D$ with the following properties: for all $m,n\in\N_{\geq 1}$
\begin{itemize}
\item[(i)] if $m<n$, then
\begin{align*}
f(d_m)(1+\frac{d_m^{-2}}{m+\frac{1}{m}}) &<f(d_n)(1+\frac{d_n^{-2}}{n+\frac{1}{n}})\textrm{ and }\\
f(d_n)(1+\frac{d_n^{-2}}{n})&<f(d_m)(1+\frac{d_m^{-2}}{m})<2,
\end{align*}
\item[(ii)] if $d \in D$, $n\geq 2$ and $d_1\leq d_{n-1}^7<d<d_n$, then
 \begin{equation*}
 f(d)(1+d^{-2}) < f(d_n) \textrm{ or } f(d)>f(d_n)(1+d_n^{-2}).
 \end{equation*}
 \item[(iii)] $d_n > d_{n-1}^{49}$ for $n\geq 2$.
 \end{itemize}
\end{fact}
It is in the construction of this sequence that the density of $f(D)$ is used. In particular, property (i) in Fact \ref{factdisc} depends crucially on this assumption. It is worth pointing out that so far we haven't said anything about the definability of the range of the sequence. For the following, fix a sequence $(d_n)_{n=1}^{\infty}$ of elements in $D$ given by Fact \ref{factdisc}.

\begin{lem}\label{mainlemma} For every $n \in \N$, there is $c \in K$ such that
\begin{itemize}
\item[(1)]$\nu_c(A_c \cap [d_2,d_n]) \subseteq \bigcup_{m \in [2,n]} (m,m+\frac{1}{m})$ and
\item[(2)] $|\nu_c(A_c \cap [d_2,d_n]) \cap (m,m+\frac{1}{m})| = 1$ for $m\in [2,n]$.
\end{itemize}\end{lem}
\begin{proof} Take $n \in \N$. Let $c\in K$ such that
$$
f(d_{n})(1+\frac{d_{n}^{-2}}{n+\frac{1}{n}})< c < f(d_{n})(1+\frac{d_{n}^{-2}}{n}).
$$
It is left to show that
\begin{itemize}
\item[(I)] for every $m \in [2,n]$, $\nu_c(d_m) \in (m,m+\frac{1}{m})$, and
\item[(II)] $A_c \cap [d_2,d_n] = \{ d_m \ : \ m \in [2,n]\}$.
\end{itemize}
We proceed as in \cite[p. 2167]{discrete}. For (I), let $m\leq n$. By Fact \ref{factdisc}(i),
\begin{equation}\label{nuineq}
f(d_m)(1+\frac{d_m^{-2}}{m+\frac{1}{m}})<c<f(d_m)(1+\frac{d_m^{-2}}{m}).
\end{equation}
After rearrangements, \eqref{nuineq} is equivalent to the statement $\nu(d_m) \in (m,m+\frac{1}{m})$.\\
For (II), let $d \in A_c \cap [d_2,d_n]$. For a contradiction, suppose there is $m \in \mathbb{N}_{>2}$ such that $d_{m-1} < d < d_{m}$.  By \eqref{nuineq}, we have
\begin{equation*}
f(d_{m-1}) < c < f(d_{m-1})(1+d_{m-1}^{-2}).
\end{equation*}
Since $\varphi(d,c)$ holds, $d_{m-1}<d^{\frac{1}{7}}$. Hence $d_{m-1}^7<d<d_{m}$.
Thus by Fact \ref{factdisc}(ii),
\begin{equation*}
f(d)(1+d^{-2}) < f(d_{m})\textrm{ or }f(d)>f(d_m)(1+d_m^{-2}).
\end{equation*}
By the definition of $c$ and (I), we get
\begin{equation*}
f(d)(1+d^{-2})< c\textrm{ or } f(d) > c .
\end{equation*}
Hence the inequality $f(d)< c < f(d)(1+d^{-2})$ fails and thus $d \notin A_c$. \newline
Let $m \in \mathbb{N}_{>1}$ and $m\leq n$. We have to show that $d_m \in A_c$. By \eqref{nuineq}, it only remains to establish $\varphi(d_m,c)$. Therefore let $d \in D$ with $d_m^{\frac{1}{7}}<d < d_m$. By Fact \ref{factdisc}(iii), $d_{m-1}^{49}<d_m$. Hence $d_{m-1}^7< d$. As in the above argument, we get that $f(d)< c < f(d)(1+d^{-2})$ does \emph{not} hold. Hence $\varphi(d_m,c)$ and $d_m\in A_c$. Thus (II) holds.
\end{proof}

\begin{proof}[Proof of Theorem A] Let $S \subseteq K_{>0} \times K^3$ be
$$
\{ (a,b_1,b_2,b_3) \in K_{>0} \times K^3 : b_2,b_3 \in A_{b_1} \wedge a+ \nu_{b_1}(b_2) \in \nu_{b_1}(A_{b_1} \cap [b_2,b_3])\}.
$$
We will now show that $S$ satisfies the assumption of Lemma \ref{asym}. Let $n \in \N$ and $\varepsilon \in K\cap[0,1/2]$. Choose $N \in \N_{>1}$ so large such that $N^{-1} < \varepsilon$. By Lemma \ref{mainlemma}, there is $c \in K$ such that
$$\nu_c(A_c \cap [d_N,d_{N+n}]) \subseteq \bigcup_{m \in [N,N+n]} (m,m+\frac{1}{m})$$
and $$|\nu_c(A_c \cap [d_N,d_{N+n}]) \cap (m,m+\frac{1}{m})| = 1, \textrm{ for } m \in [N,N+n].$$ Since $N^{-1}<\varepsilon$, we get that $$S_{(c,d_1,d_2)} \subseteq \bigcup_{m \leq n} (m-\varepsilon,m+\varepsilon)$$ and $|S_{(c,d_1,d_2)} \cap (m-\varepsilon,m+\varepsilon)| = 1$ for  $m \leq n$.
\end{proof}

 

\end{document}